\def\to{\rightarrow}

\def\be{\begin{equation}}

\def\bs{\bigskip}

\def\BB{{\mathcal B}}

\def\DD{{\mathcal{D}}}

\def\HH{{\mathcal H}}

\def\LL{{\mathcal{L}}}

\def\Tt{\tilde{U}}
\def\KKt{\tilde{\KK}}

\def\Th{{\Theta}}

\def\Hb{{\textbf{H}}}
\def\Tb{{\textbf{T}}}

\def\KK{{\mathcal K}}

\def\C{{\mathbb{C}}}
\def\D{{\mathbb{D}}}
\def\Z{{\mathbb{Z}}}

\def\T{{\mathbb T}}

\documentstyle[10pt,amssymb]{amsart}

\title{On a connection between Naimark's dilation theorem, spectral representations, and characteristic functions}
\author{Mishko Mitkovski}

\address{Texas A\&M University
\\ Department of Mathematics\\
College Station, TX 77843, USA } \email{mmitkov@@math.tamu.edu}

\theoremstyle{plain}
\newtheorem{theorem}{Theorem}[section]

\newtheorem{corollary}[theorem]{Corollary}

\theoremstyle{definition}

\newtheorem{remark}{Remark}

\newenvironment{proof}
{{\noindent\it Proof:}}{\hfill$\Box$}

\numberwithin{equation}{section}

\begin{document}
\begin{abstract}
We give a Herglotz-type representation of an arbitrary generalized spectral measure. As an application, a new proof of the classical Naimark's dilation theorem is given. The same  approach is used to describe the spectrum of all unitary rank-one perturbations of a given partial isometry. 
\end{abstract}

\maketitle

\section{Introduction}

Let $\BB$ be a family of all Borel sets on the unit circle $\T$. By a generalized spectral measure on $\T$ we mean a function $B:\BB\to\BB(\HH)$ whose values are positive bounded self-adjoint operators on $\HH$ such that $B(\emptyset)=0$, $B(\T)=I$ and for every sequence $\Delta_1,\Delta_2,...$ of mutually disjoint Borel sets, we have $$B(\Delta_1\cup\Delta_2\cup...)=\sum_{i=1}^{\infty}B(\Delta_i)$$ in the strong operator topology. If we require all the values to be orthogonal projections then we have an (ordinary) spectral measure. A classical theorem of Naimark~\cite{Naimark1} says that any generalized spectral measure can be represented as a projection of an ordinary spectral measure. This theorem is considered by many as the beginning of Dilation Theory. Since then many different proofs and generalization have appeared (e.g.~\cite{Naimark2, Foias, Stinespring, Paulsen}). We propose yet another approach, which involves in a natural way characteristic functions and spectral representations of unitary operators; it also relates Naimark's theorem for the first time to the subject of rank-one perturbations of a given operator. The latter is another classical subject with a rich literature behind (see~\cite{Poltoratski}  and the references therein). The key idea in our approach is to obtain a representation for a generalized spectral measure which is reminiscent to the well known one:

\[\left< (U+zI)(U-zI)^{-1} h_1 | h_2 \right> = \int_\T \frac{\xi+z}{\xi -z } \; d\left<(E(\xi)h_1 | h_2 \right>\]
 relating a unitary operator $U$ to its spectral measure $E$. We will show that if $\KK$ is a closed subspace of $\HH$ and the generalized spectral measure $B:\BB\to\BB(\KK)$ is obtained from a spectral measure $E:\BB\to\BB(\HH)$ by $B(\Delta)=P_\KK E(\Delta)$ then the above mentioned representation is given by:          

$$\left< (\Tt+\Th_S(z))(\Tt-\Th_S(z))^{-1} k_1 | k_2 \right>=\int_\T \frac{\xi+z}{\xi -z } \; d\left<(B(\xi)k_1 | k_2 \right>$$
where $\Tt:U^*(\KK) \to \KK$ is the restriction of $P_{\KK}U$ on $U^*(\KK)$ and $\Th_S(z)$ is the characteristic function corresponding to the operator $S:=(I-P_{\KK})U$. In the case when $\KK$ is one-dimensional this relation reduces to a relation that D.~Clark~\cite{Clark} used in his treatment of the rank-one perturbations of a restricted shift. This is the point where the connection between these seemingly unrelated concepts is made. 

The  paper is organized as follows. The next section presents some preliminary material. Section 3 contains the main theorem which contains the above mentioned representation. As an application, a new proof of Naimark's dilation theorem is given. In the last section an application to the rank-one perturbations of a partial isometry is given.  

\section{Preliminaries}

We recall the basic facts from the the model theory of completely nonunitary contractions developed by Sz. Nagy and Foias~\cite{Foias}. Let $T \in \BB(\HH)$ be a contraction, $\|T\| \leq 1$. The self adjoint operators $D_T:=(I-T^*T)^{1/2}$ and $D_{T^*}:=(I-TT^*)^{1/2}$ are called the defect operators of $T$ and the spaces $\DD_T:=\overline{D_T\HH}$, $\DD_{T^*}:=\overline{D_{T^*}\HH}$ are called the defect spaces for $T$. The defect indices are defined by $\partial_T:=\text{dim}\DD_T$ and $\partial_{T^*}:=\text{dim}\DD_{T^*}$. These indices measure, in a certain sense, how much a contraction differs from a unitary  operator. If $T$ is a partial isometry then the defect operators $D_T$ and $D_{T^*}$ are orthogonal projections onto the initial and the final space of $T$ respectively. 

The characteristic function of a contraction $T$ is an operator-valued function $\Th_T(z):\DD_T \to \DD_{T^*}$ defined by: $$ \Th_T(z):=-T+z D_{T^*}(I-zT^*)^{-1}D_T|\DD_{T}.$$ It is always an analytic contraction-valued function. A contraction-valued function $\Th:\D \to \BB(\LL,\LL^*)$ is called pure if $\|\Th(0)l\|<\|l\|$ for $l\in \LL$, $l\neq 0$. Two contraction-valued analytic functions $\Th_1: \D \to \BB(\LL_1,\LL_1^*)$, $\Th_2: \D \to \BB(\LL_2,\LL_2^*)$ are said to coincide if there exist unitaries $\omega:\LL_1 \to \LL_2$, $\omega_*:\LL_1^* \to \LL_2^*$ such that $\Th_1(z)={\omega_*}^{-1} \Th_2(z) \omega$. There is an easy characterization of partial isometries in terms of their characteristic functions. Namely, a contraction $T$ is a partial isometry if and only if $\Th_T(0)=O$.

A contraction $T$ is said to be completely nonunitary (c.n.u.) if it is not unitary on any of its invariant subspaces. There always exists a unique decomposition $\HH=\HH_0 \oplus \HH_1$ into subspaces reducing $T$, such that $T_{|\HH_0}$ is unitary and $T_{|\HH_1}$ is c.n.u..  
 
One of the main theorems of Sz. Nagy and Foias theory is the fact that every c.n.u. contraction $T$ on a separable Hilbert space is unitarily equivalent to an operator $\Tb$ acting on the space $$\Hb=[H^2(\DD_{T^*})\oplus\overline{\Delta_T L^2(\DD_T)}]\ominus\{\Th_Tu\oplus\Delta_T u: u\in H^2(\DD_T)\},$$  
by $$\Tb(u\oplus v)=\textbf{P}_{\Hb}(zu\oplus zv),$$ where $\textbf{P}_{\Hb}$ is the orthogonal projection of $H^2(\DD_{T^*})\oplus\overline{\Delta_T L^2(\DD_T)}$ onto $\Hb$ and $\Delta_T(t):=(I-\Th_T^*(e^{it})\Th_T(e^{it}))^{1/2}.$
Here, as usual, $L^2(\DD_T)$ is the space of $\DD_T$-valued square integrable functions on $\T$ and $H^2(\DD_T)$ is the corresponding Hardy space.

The following fact will be important for us. For any given contraction-valued analytic function $\Th: \D \to \BB(\LL,\LL^*)$ which is pure, there exists a contraction $\Tb$ whose characteristic function coincides with $\Th(z)$.

In the case when $\Th_T(z)$ is an inner function (i.e. $\Th_T(\zeta)$ is an isometry for a.e. $\zeta \in \T$) then $\Hb$ has a simpler form $\Hb=H^2(\DD_{T^*})\ominus\{\Th_Tu: u\in H^2(\DD_T)\}.$ This happens if and only if $T^{*n}\to O$ as $n\to \infty$. If, in addition, the defect indices of $T$ are both equal to 1 then $\Hb=H^2(\D)\ominus\Th_T H^2(\D)$ is a space of analytic scalar-valued functions. In this case $\Th_T(z)$ can be viewed as a usual (scalar-valued) inner function. In~\cite{Clark}, all the unitary rank-one perturbations of this type of operators $T$ are examined. It is shown there that all the rank-one unitary perturbations can be parametrized by points $\alpha \in \T$ and that the spectral measures $\sigma_{\alpha}$ for each of these operators can be easily obtained from the characteristic function $\Th_T(z)$. Namely, they are determined by: $$\frac{\alpha+\Th_T(z)}{\alpha+\Th_T(z)}=\int_{\T}\frac{\xi+z}{\xi-z}d\sigma_{\alpha}(\xi).$$ The collection of measures that are associated to a given inner function (or more generally to an analytic self-map of the disc) in this way are called Clark (or Aleksandrov-Clark) measures~\cite{Poltoratski}.

\section{Main Results}
       
Let $U: \HH \to \HH$ be a unitary operator. Fix $h\in \HH$. Then the function $\left< (U+zI)(U-zI)^{-1} h | h \right>$ is  holomorphic in $z \in \D$ with positive real part. Thus, there exists a unique measure $\mu_h$ on the unit circle $\T$ such that

\[\left< (U+zI)(U-zI)^{-1} h | h \right> = \int_\T \frac{\xi+z}{\xi -z } \; d\mu_h(\xi).\]
Using polarization, for any $h_1, h_2\in \HH$ there exists a measure $\mu_{h_1, h_2}$ such that:
 
\[\left< (U+zI)(U-zI)^{-1} h_1 | h_2 \right> = \int_\T \frac{\xi+z}{\xi -z } \; d\mu_{h_1, h_2}(\xi).\]
Let $\Delta$ be any Borel set on the unit circle $\T$. Then $\mu_{h_1, h_2}(\Delta)$ is a skew-symmetric function of $h_1$ and $h_2$, linear in $h_1$, and bounded by $\|h_1 \| \|h_2\|$. Therefore, it can be represented as
\[ \mu_{h_1, h_2}(\Delta) = \left<E(\Delta)h_1 | h_2 \right>,\]
for some positive bounded operator $E(\Delta)$. It is well known that $E: \BB \to \BB(\HH)$ is an ordinary spectral measure. 

Let $\KK \subset \HH$ be a closed subspace and let $P_{\KK}:\HH\to\KK$ be the orthogonal projection onto $\KK$. Define $T,S : \HH \to \HH$ by $T=P_{\KK}U$ and $S=(I-P_{\KK})U$, respectively. Then $T$ and $S$ are partial isometries and the characteristic function $\Th_S(z):U^*(\KK)\to \KK$ satisfies $\Th_S(0)=O$ (and hence is pure). It will be useful to denote by $\Tt$ the restriction of $U$ on the closed subspace $U^*(\KK)$, and view it as an operator from $U^*(\KK)$ to $\KK$. Obviously, $\Tt u=Tu=Uu$ for any vector $u \in U^*(\KK)$ and $\Tt^*k=T^*k=U^*k$ for any $k \in \KK$. Observe also that $\Tt^*P_\KK=T^*P_\KK=U^*P_\KK^2=T^*$. 

For fixed $k\in \KK$, again $\left< (\Tt+\Th_S(z))(\Tt-\Th_S(z))^{-1} k | k\right>$ is a holomorphic function in $z \in \D$ with positive real part and consequently there exists a unique measure $\sigma_{k}$ satisfying $$\left< (\Tt+\Th_S(z))(\Tt-\Th_S(z))^{-1} k| k \right>=\int_\T \frac{\xi+z}{\xi -z } \; d\sigma_{k}(\xi).$$ By polarization again, for any $k_1, k_2 \in \KK$ there exists a measure $\sigma_{k_1, k_2}$ such that:
$$\left< (\Tt+\Th_S(z))(\Tt-\Th_S(z))^{-1} k_1 | k_2 \right>=\int_\T \frac{\xi+z}{\xi -z } \; d\sigma_{k_1, k_2}(\xi).$$
For any Borel set $\Delta$ on the unit circle let $B(\Delta)$ be the positive self-adjoint operator satisfying \[ \sigma_{k_1, k_2}(\Delta) = \left<B(\Delta)k_1 | k_2 \right>,\] for all $k_1, k_2 \in \KK$. It can be shown that $B: \BB \to \BB(\KK)$ is a generalized spectral measure.

\begin{theorem} \label{main} For any $k\in \KK$ the following equality holds:

$$\left< (U+zI)(U-zI)^{-1} k | k \right> =\left< (\Tt+\Th_S(z))(\Tt-\Th_S(z))^{-1} k| k \right>.$$

\end{theorem} 

\begin{proof}  The left- and right-hand side of the equality we are aiming to prove are equal to $2 \left< (I-\Th_S(z)\Tt^*)^{-1} k | k \right> - \|k\|^2$ and
$2 \left< (I-zU^*)^{-1} k | k \right> - \|k\|^2$, respectively. Hence, it is equivalent to show that:
\[ \left< (I-\Th_S(z)\Tt^*)^{-1} k | k \right> = \left< (I-zU^*)^{-1} k | k \right>. \] 
To prove this equality, we first notice that
\begin{align*}  \left< (I-\Th_S(z)\Tt^*)^{-1} k | k \right>  &= \left< (I-z P_\KK (I-zS^*)^{-1} \Tt^*)^{-1} k | k \right> \\
& = \left< (\Tt\Tt^*-z P_\KK (I-zS^*)^{-1} \Tt^*)^{-1} k | k \right> \\ &= \left< \Tt (\Tt-z P_\KK (I-zS^*)^{-1})^{-1} k | k \right> \\
&= \left< \Tt(\Tt(I-zS^*)(I-zS^*)^{-1}-z P_\KK (I-zS^*)^{-1})^{-1} k | k \right> \\
&= \left< \Tt(I-zS^*) (\Tt(I-zS^*)-z P_\KK)^{-1} k | k \right> \\
&=\left< \Tt(I-zS^*) (\Tt(I-zS^*)-z \Tt\Tt^* P_\KK)^{-1} k | k \right> \\
&= \left< \Tt (I-zS^*) (I-zS^*-z \Tt^* P_\KK)^{-1} \Tt^* k | k \right>.\end{align*}

Now, recalling that $\Tt^*P_\KK =T^*$, $T^*k=\Tt^*k$ we obtain

\begin{align*}& \left< (I-\Th_S(z)\Tt^*)^{-1} k | k \right> = \left< (I-zS^*) (I-zS^*-z T^*)^{-1} \Tt^* k | \Tt^* k \right>\\
&= \left< (I-zS^*-z T^*)^{-1} T^* k | T^* k \right> - \left<z (I-zS^*-z T^*)^{-1} T^* k | S T^* k \right> \\
&= \left< (I-zS^*-z T^*)^{-1} T^* k |T^* k \right>.\end{align*}

Since $ST^*k =0$, we finally have

\begin{align*}\left< (I-\Th_S(z)\Tt^*)^{-1} k | k \right> & =\left< (I-z(T+S)^*)^{-1} T^*k |T^* k \right> \\ 
&=\left< (I-zU^*)^{-1}U^* k |U^* k \right> \\ 
&=\left< U^*(I-zU^*)^{-1} k |U^* k \right> \\
&=\left< (I-zU^*)^{-1} k |k \right>.\end{align*}

\end{proof}

\begin{remark}
If $\KK=\HH$ then $S=O$ on $\HH$ and hence $\Th_S(z)=zI$. Also, clearly $\Tt=U$. Thus, we can view $(\Tt+\Th_S(z))(\Tt-\Th_S(z))^{-1}$ as a substitution for $(U+zI)(U-zI)^{-1}$ in the case when $\KK$ is a true subspace.
\end{remark}

\begin{corollary} \label{cor} For any Borel set $\Delta \subset \T$, $B(\Delta)=P_\KK E(\Delta)$.
\end{corollary}

\begin{proof}
To show the claim it is equivalent to show 
\[ \left< P_\KK E(\Delta)k_1 | k_2 \right> =  \left< B(\Delta)k_1 | k_2 \right> \] for all $k_1, k_2 \in \KK$. Since $P_\KK$ is an orthogonal projection, this is equivalent to: \[ \left<E(\Delta)k_1 | k_2 \right> =  \left< B(\Delta)k_1 | k_2 \right>. \] Therefore, using polarization it suffices to prove \[ \left<E(\Delta)k | k \right> =  \left< B(\Delta)k | k \right> \] for all $k \in \KK$. 
But, this easily follows from Theorem~\ref{main}.

\end{proof}

Next we give a new proof of Naimark's dilation theorem.
\begin{theorem}[Naimark~\cite{Naimark1}] Let $B: \BB \to \BB(\KK)$ be a generalized spectral measure. Then there exist $\HH \supset \KK$ and an ordinary spectral family $E: \BB \to \BB(\HH)$ such that for any Borel set $\Delta \subset \T$, $B(\Delta)=P_\KK E(\Delta)$.
\end{theorem}

\begin{proof} Let $k_1, k_2 \in \KK$. Define $\sigma_{k_1, k_2} (\Delta) = \left< B(\Delta) k_1| k_2 \right>$ for any Borel set $\Delta \subset \T$. Clearly, $\sigma_{k_1, k_2}$ is a Borel measure on $\T$. For any $z \in \D$ define $F(z): \KK \to \KK$ such that
\[\left< F(z)k_1| k_2 \right> = \int_{\T} \frac{\xi+z}{\xi-z} \; d\sigma_{k_1, k_2}(\xi)\]
for every $k_1, k_2 \in \KK$. It is easy to see that $F(z)$ is a linear operator with a positive real part and $F(0) =I$. It is also analytic as an operator-valued function in $z \in \D$. The inverse $(I + F(z))^{-1}$ is defined on a dense subset of $\KK$ since $F(z)$ has a positive real part. Set
$\Th(z)=(F(z)-I)(F(z)+I)^{-1}$ on that dense set. By continuity, $\Th(z)$ can be extended on the whole $\KK$ with $\|\Th(z) \| \leq 1$. Moreover, $\Th(z)$ is an analytic contraction-valued function with $\Th(0) = O$ (and hence is pure). Thus, there exists a completely non-unitary contraction $S: \HH_1 \to \HH_1$ whose characteristic function $\Th_S(z)$ coincides with $\Th(z)$. More precisely, there exist unitary operators $\omega: \DD_S \to \KK$ and $\omega_*: \DD_{S^*} \to \KK$ such that $\Th_S(z)={\omega_*}^{-1} \Th(z) \omega$. Define $\Tt={\omega_*}^{-1}  \omega: \DD_S \to \DD_{S^*}$ and set $T: \HH_1 \to \HH_1$ to be \[ Th =\begin{cases} \Tt h, &h \in \DD_S \\
0, &h \in \text{Ker} D_S. \end{cases}\]
Finally, define $U_1 = T+S: \HH_1 \to \HH_1$. Clearly, $U_1$ is unitary. Let $E_1(\Delta)$ be the spectral measure corresponding to $U_1$ and let $B_1(\Delta)$ be the generalized spectral measure such that
\[\left< (\Tt +\Th_S(z)) (\Tt - \Th_S(z))^{-1} s_1| s_2 \right> = \int_{\T} \frac{\xi+ z}{\xi -z}\; d\left<B_1(\xi)s_1|s_2\right>\]
for every $s_1, s_2 \in \DD_{S^*}$. By the previous Corrolary, for any Borel set $\Delta$ we have
\[ B_1(\Delta)=P_{\DD_{S^*}} E_1(\Delta). \]
Now, since
\[\left< (\Tt +\Th_S(z)) (\Tt - \Th_S(z))^{-1} s_1| s_2 \right> =\left< (I +\Th(z)) (I - \Th(z))^{-1} \omega_* s_1| \omega_* s_2 \right>\] and $F(z)=(I +\Th(z)) (I - \Th(z))^{-1}$, it follows that 
$B(\Delta)= \omega_* B_1(\Delta)\omega_*^{-1}$. 

Define $\HH = \KK \oplus \text{Ker}D_{S^*}$ and $W_* = \omega_* \oplus I: \HH_1 \to \HH$. Clearly, $W_*$ is unitary. Then
$ B(\Delta) = P_\KK E(\Delta)$ where $P_\KK = W_* P_{\DD_{S^*}} W_*^{-1}$ and $E(\Delta) = W_* E_1(\Delta) W_*^{-1}$. One readily sees that $P_\KK:\HH \to \HH$ defined this way is the orthogonal projection onto $\KK$ and $E(\Delta)$ is a spectral measure on $\HH$. 

\end{proof}

\section{Rank-one perturbations of partial isometries} 

Let $S:\HH \to \HH$ be a c.n.u. partial isometry with both defect indices equal to 1, i.e., $\partial_S= \partial_{S^*}=1$. Denote by $\KK$ and $\KKt$ the orthogonal complements of the final and the initial space of $S$, respectively. Fix two unit vectors $k \in \KK$ and $\tilde{k} \in \KKt$. For any   complex number $\alpha$ of modulus 1,  define the linear operator $U_\alpha$ by \[U_\alpha h = \begin{cases} Sh,  &\text{ if $h \perp \KKt$} \\
\alpha k, & \text{ if $h = \tilde{k}$}. \end{cases} \] Clearly, $U_\alpha$ are unitary operators; these are the only unitary rank-one perturbations of $S$. Consider $\KK_{\alpha}:=\overline{\text{span}}\{U_{\alpha}^n k : n \in \Z\}$. One can show by induction that for any $h\in \KK_{\alpha}^{\perp}$ and $\beta \in \T$ we have that $U_{\beta}^n h=S^n h$ and $U_{\beta}^{*n} h=S^{*n} h$. Therefore,  $\KK_\beta^\perp \subset \KK_\alpha^\perp$ and, by symmetry, $\KK_\beta^\perp = \KK_\alpha^\perp$. This space $\KK_\alpha^\perp$  is reducing for $S$ and $S$ is unitary there. Since $S$ is c.n.u., we have that $\KK_\alpha =\HH$ for all $\alpha \in \T$.

Next we will describe the spectrum of $U_\alpha$. Let $\Th_S(z)$ be the characteristic function of $S$. For each $z \in \D$, $\Th_S(z)$ is an operator between two one-dimensional spaces $\KKt$ and $\KK$. Denote by $\phi(z)$ the scalar valued analytic function such that $\Th_S(z)\tilde{k}=\phi (z)k$ for each $z\in\C$. Clearly, $|\phi(z)| \leq 1$. Define also $\Tt_\alpha:\KKt \to \KK$ to be the operator sending $\tilde{k}$ to $\alpha k$. Let $\sigma_{\alpha}$ be the measure on $\T$ for which $$\left< (\Tt_\alpha+\Th_S(z))(\Tt_\alpha-\Th_S(z))^{-1} k | k \right>=\int_\T \frac{\xi+z}{\xi -z } \; d\sigma_{\alpha}(\xi).$$ This is equivalent to $$\frac{\alpha+\phi (z)}{\alpha-\phi (z)}=\int_\T \frac{\xi+z}{\xi -z } \; d\sigma_{\alpha}(\xi),$$
which implies that $\sigma_{\alpha}$ is a Clark measure for $\phi(z)$. It is also immediate that the corresponding generalized spectral measure $B_{\alpha}(\Delta)$ is simply given by $B_{\alpha}(\Delta)k=\sigma_{\alpha}(\Delta) k$. If $E_{\alpha}(\Delta)$ is the spectral measure corresponding to $U_{\alpha}$ then it follows from Corollary~\ref{cor} that $$\left< E_{\alpha}(\Delta)k | k  \right>=\sigma_{\alpha}(\Delta).$$ Since $\overline{\text{span}}\{E_\alpha(\Delta)k : \Delta \text{ Borel subset of } \T\}=\HH$ the last equality proves the following:

\begin{corollary} \label{cor1} The spectrum of $U_{\alpha}$ coincides with the support of $\sigma_{\alpha}$. Thus, it consists of the union of those points in $\T$ at which $\phi(z)$ cannot be analytically continued and those $\zeta \in \T$ at which $\phi(z)$ is analytically continuable with $\phi(\zeta)=\alpha$. The set of eigenvalues of $U_{\alpha}$ coincides with the set of all the atoms of $\sigma_{\alpha}$.  
\end{corollary}

\begin{remark} There are several well-known conditions describing the atoms of a Clark measure $\sigma_{\alpha}$. An important one (goes back to M. Riesz) is the following: $\sigma_{\alpha}$ has a point mass at $\zeta$ if and only if $\phi(z)$ has the nontangential limit $\alpha$ at $\zeta$ and for all (or one) $\beta$ in $\T$ different from $\alpha$, $$\int_{\T}\frac{d\sigma_{\beta}(\zeta)}{|\xi-\zeta|^2}<\infty.$$  
\end{remark}

\begin{remark} In~\cite{Clark}, a similar description (although with different methods) of the spectra is obtained for unitary rank-one perturbations of a restricted shift. Notice that not every partial isometry can be represented as a restricted shift on the space $H^2(\D) \ominus \Th H^2(\D)$ of scalar-valued functions. Thus, the proposition above is not implied by the results in~\cite{Clark}. 
\end{remark}

Finally, we can also consider the more general case when $S$ is c.n.u. with equal (possibly infinite) defect indices. Let again $\KK$ and $\KKt$ be the orthogonal complements of the final and the initial space of $S$, respectively. For a unitary operator $A: \KKt \to \KK$, similarly as in~\cite{Ball}, we can define a unitary perturbation $U_A$ of $S$ by \[U_A h = \begin{cases} Sh,  &\text{ if $h \perp \KKt$} \\
Ah, & \text{ if $h \in \KKt$}. \end{cases} \]

As in the case of the rank-one perturbations, one can show that $S$ c.n.u. implies $\overline{\text{span}}\{U_A^nk : k \in \KK, n \in \Z \} =\HH$. To describe the spectrum of $U_A$ notice that by Theorem~\ref{main} we have 
\begin{align*}\left< (A+\Th_S(z))(A-\Th_S(z))^{-1} k | k \right> &=\left< (\Tt_A+\Th_S(z))(\Tt_A-\Th_S(z))^{-1} k | k \right>\\ &=\int_\T \frac{\xi+z}{\xi -z } \; d\left< E_A(\xi)k|k \right>.\end{align*}
Corollary~\ref{cor1} has the following analogue in this general case.

\begin{corollary} The spectrum of $U_A$ consists of the union of those points in $\T$ at which $\Th_S(z)$ cannot be analytically continued and those $\zeta \in \T$ at which $\Th_S(z)$ is analytically continuable with $\Th_S(\zeta)-A$ not invertible.
\end{corollary}

Further extensions of these results will be considered in a subsequent paper.

 \bs


\begin{thebibliography}{9}
\bibitem{Ball}{\sc Ball, J. A. and Lubin, A.} {\it On a class of contractive perturbations of restricted shifts,} Pacific J. Math., \textbf{63} (1976), 309 -323.

\bibitem{Clark}{\sc Clark, D. N.} {\it One dimensional perturbations of restricted shifts,} J. Analyse Math., \textbf{25} (1972), 169-191. 

\bibitem{Naimark1}{\sc Naimark, M. A.}{ \it Spectral functions of a symmetric operator,} Izv. Akad. Nauk. SSSR Ser. Mat., \textbf{4} (1940), 277-318.

\bibitem{Naimark2}{\sc Naimark, M. A.}{ \it On a representation of additive operator set functions,} Dokl. Acad. Sci. SSSR, \textbf{41} (1943), 373-375.

\bibitem{Poltoratski} {\sc Poltoratski,~A. and Sarason, D.} {\it Aleksandrov-{C}lark measures,} Recent advances in operator-related function theory, Contemp. Math., \textbf{393} (2006), 1-14.


\bibitem{Foias}{\sc Sz.-Nagy, B. and Foias, C.}  {\it Harmonic analysis of operators on {H}ilbert space,} Translated from the French and revised, North-Holland Publishing Co., Amsterdam, 1970.

\bibitem{Paulsen} {\sc Paulsen, V.} {\it Completely Bounded Maps and Operator Algebras,} Cambridge University Press, 2003.

\bibitem{Stinespring} {\sc Stinespring, W. F.} {\it Positive Functions on C*-algebras,} Proc.  Amer. Math. Soc., \textbf{6} (1955), 211-216.

 \end{thebibliography}
\end{document}